\documentclass[a4paper,12pt]{amsart}
\usepackage{}
\usepackage{mathrsfs}

\usepackage{amssymb}
\usepackage{latexsym}
\usepackage{amsfonts}
\usepackage{amsmath}
\usepackage{eucal}
\usepackage{bm}
\usepackage{bbm}
\usepackage{graphicx}
\usepackage[english]{varioref}
\usepackage[nice]{nicefrac}
\usepackage[all]{xy}
\usepackage{amsthm}

\newcommand{\Ad}{\text {\rm Ad}}

\def\i{^{-1}}
\def\ge{\geqslant}
\def\le{\leqslant}
\def\<{\langle}
\def\>{\rangle}

\def\d{\text{d}}

\def\a{\alpha}
\def\b{\beta}

\def\G{\Gamma}
\def\d{\delta}

\def\e{\epsilon}

\def\s{\sigma}
\def\t{\tau}
\def\th{\theta}
\def\k{\kappa}
\def\l{\lambda}

\def\Om{\Omega}

\def\ZZ{\mathbb Z}
\def\NN{\mathbb N}

\def\RR{\mathbb R}

\def\ca{\mathcal A}

\def\ch{\mathcal H}
\def\ci{\mathcal I}

\def\co{\mathcal O}

\def\tH{\tilde H}

\def\tu{\tilde u}
\def\tx{\tilde x}
\def\ty{\tilde y}

\def\tW{\tilde W}
\def\tw{\tilde w}

\def\fH{\mathfrak H}

\def\fR{\mathfrak R}

\def\fs{\mathfrak S}

\theoremstyle{plain}
\newtheorem{thm}{Theorem}[section]
\newtheorem*{thm*}{Theorem}
 \newtheorem{prop}[thm]{Proposition}
 \newtheorem{lem}[thm]{Lemma}
 \newtheorem{cor}[thm]{Corollary}

\theoremstyle{definition}

\newtheorem{example}[thm]{Example}

\newtheorem*{thm1}{Theorem A}
\newtheorem*{thm2}{Theorem B}

\newtheorem*{thm4}{Theorem C}

\theoremstyle{remark}

\newtheorem*{rmk}{Remark}
\newtheorem*{claim*}{Claim}

\begin{document}
\author{Xuhua He}
\address{Department of Mathematics and Institute for advanced Study, The Hong Kong University of Science and Technology, Clear Water Bay, Kowloon, Hong Kong}
\email{xuhuahe@gmail.com}
\thanks{X.H. was partially supported by HKRGC grant 602011.}
\author{Sian Nie}
\address{Max Planck Institute for mathematics, Vivatsgasse 7, 53111, Bonn, Germany}
\email{niesian@amss.ac.cn}
\title[]{$P$-alcoves, parabolic subalgebras and cocenters of affine Hecke algebras}

\maketitle

\begin{abstract}
This is a continuation of the sequence of papers \cite{HN2}, \cite{H99} in the study of the cocenters and class polynomials of affine Hecke algebras $\ch$ and their relation to affine Deligne-Lusztig varieties. Let $w$ be a $P$-alcove element, as introduced in \cite{GHKR} and \cite{GHN}. In this paper, we study the image of $T_w$ in the cocenter of $\ch$. In the process, we obtain a Bernstein presentation of the cocenter of $\ch$. We also obtain a comparison theorem among the class polynomials of $\ch$ and of its parabolic subalgebras, which is analogous to the Hodge-Newton decomposition theorem for affine Deligne-Lusztig varieties. As a consequence, we present a new proof of \cite{GHKR} and \cite{GHN} on the emptiness pattern of affine Deligne-Lusztig varieties.
\end{abstract}



\section*{Introduction}

\subsection{} The purpose of this paper is twofold. We use some ideas arising from affine Deligne-Lusztig varieties to study affine Hecke algebras, and we apply the results on affine Hecke algebras to affine Deligne-Lusztig varieties. 

For simplicity, we only discuss the equal-parameter case in the introduction. The case of unequal parameters and the twisted cocenters will also be presented in this paper.

Let $\fR=(X, R, Y, R^\vee, F_0)$ be a based root datum and let $\tW$ be the associated extended affine Weyl group. An affine Hecke algebra $\ch$ is a deformation of the group algebra of $\tW$. It is a free $\ZZ[v, v \i]$-algebra with basis $\{T_w\}$, where $w \in \tW$. The relations among the $T_w$ are given in $\S$\ref{IMpresent}. This is the Iwahori-Matsumoto presentation of $\ch$.

The cocenter $\bar \ch=\ch/[\ch, \ch]$ of $\ch$ is a useful tool in the study of the representation theory and structure of $p$-adic groups. We will discuss some applications of the cocenter as they serve as the motivation for this paper.

Let $R(\ch)$ be the Grothendieck group of representations of $\ch$. Then the trace map $Tr: \bar \ch \to R(\ch)^*$ relates the cocenter $\bar \ch$ to the representations of $\ch$. This map was studied in \cite{BDK}, \cite{Ka}.

In \cite{HN2}, we provide a standard basis of the cocenter $\bar \ch$, which is constructed as follows. For each conjugacy class $\co$ of $\tW$, we choose a minimal length representative $w_\co$. Then the image of $T_{w_\co}$ in $\bar \ch$ is independent of the choice of $w_\co$ and the set $\{T_{w_\co}\}$, where $\co$ ranges over all the conjugacy classes of $\tW$, is a basis of $\bar \ch$. This is the Iwahori-Matsumoto presentation of $\bar \ch$.

Moreover, for any $w \in \tW$, $$T_w \equiv \sum_{\co} f_{w, \co} T_{w_\co} \mod [\ch, \ch]$$ for some $f_{w, \co} \in \NN[v-v \i]$. The coefficients $f_{w, \co}$ are called the {\it class polynomials}.

In \cite{H99}, the first-named author proved the ``dimension=degree'' theorem which relates the degrees of the class polynomials of $\ch$ to the dimensions of the affine Deligne-Lusztig varieties of the corresponding $p$-adic group $G$.

\subsection{} Let $J \subset S_0$ and let $\ch_J$ be the corresponding parabolic subalgebra of $\ch$. For a given $w \in \tW$, we would like to express $T_w$ as an element in $\ch_J+[\ch, \ch]$ for some $J$. 

This is useful for the representation theory because a large number of the representations of $\ch$ are built on the parabolically induced representations $\text{Ind}^\ch_{\ch_J}(-)$ for some $J$. It is also useful for the study of affine Deligne-Lusztig varieties as one would like to compare the affine Deligne-Lusztig varieties for $G$ and for the Levi subgroups of $G$.

We prove that

\begin{thm1}
Let $P$ be a (semistandard) parabolic subgroup of $G$ and let $w$ be a $P$-alcove element of type $J$. Then $T_w \in \ch_J+[\ch, \ch]$.
\end{thm1}

The notion of $P$-alcove elements was introduced by G\"ortz, Haines, Kottwitz, and Reuman in \cite{GHKR} and generalized in \cite{GHN}. Roughly speaking, $w$ is a $P$-alcove element if the finite part of $w$ lies in the finite Weyl group of $P$ and it sends the fundamental alcove to a certain region of the apartment. See \cite[Section 3]{GHKR} for a visualization.

\subsection{} Let $\co$ be a conjugacy class of $\tW$ and $w_\co$ be a minimal length element of $\co$. We may regard $w_\co$ as a $P$-alcove element for some $P$. In this case, we have a sharper result:

\begin{thm2}
Let $\co$ be a conjugacy class of $\tW$ and let $J \subset S_0$ be such that $\co \cap \tW_J$ contains an elliptic element of $\tW_J$. Then $$T_{w_\co} \equiv T^J_y  \mod [\ch, \ch]$$ for some $y \in \co \cap \tW_{J}$ of minimal length (with respect to the length function on $\tW_J$) in its $\tW_J$-conjugacy class. Here $T^J_y$ is the corresponding Iwahori-Matsumoto element in $\ch_{J}$.
\end{thm2}

The description of the element $T^J_y$ in $\ch$ uses Bernstein presentation. Thus Theorem B gives a Bernstein presentation of the cocenter $\bar \ch$. 

Notice that in the Bernstein presentation of the basis of $\bar \ch$, there are exactly $N$ elements that are not represented by elements in a proper parabolic subalgebra of $\ch$, where $N$ is the number of elliptic conjugacy classes of $\tW$. On the other hand, Opdam and Solleveld showed in \cite[Proposition 3.9]{OS} and \cite[Theorem 7.1]{OS2} that the dimension of the space of ``elliptic trace functions'' on $\ch$ also equals $N$. 
It would be interesting to relate these results via the trace map.




\subsection{} We may also compare the class polynomials of $\ch$ and of $\ch_J$ as follows:

\begin{thm4}
Let $P=z \i P_J z$ with $z \in {}^J W_0$ be a semistandard parabolic subgroup of $G$ and let $\tw$ be a $P$-alcove element. Suppose that \begin{align*}T_{\tw} &\equiv \sum_{\co} f_{\tw, \co}T_{w_\co} \mod [\ch, \ch], \\ T_{z \tw z\i}^J &\equiv \sum_{\co'} f_{z \tw z\i, \co'}^J T_{w_\co'}^J \mod [\ch_J, \ch_J].\end{align*} Then $f_{\tw, \co}=\sum_{\co' \subset \co} f_{z \tw z\i, \co'}^J.$
\end{thm4}

The Hodge-Newton decomposition theorem, which is proved in \cite[Theorem 1.1.4]{GHKR}, says that if $P=M N$ is a semistandard parabolic subgroup of $G$ and $\tw$ is a $P$-alcove element, then the corresponding affine Deligne-Lusztig varieties for the group $G$ and for the group $M$ are locally isomorphic.

Recall that there is a close relation between the class polynomials and the affine Deligne-Lusztig varieties. Thus Theorem C above can be regarded as an algebraic analog of the Hodge-Newton decomposition theorem in \cite{GHKR}.

Combining Theorem C with the ``degree=dimension'' Theorem, we can derive an algebraic proof of \cite[Theorem 1.1.2]{GHKR} and \cite[Corollary 3.6.1]{GHN} on the emptiness pattern of affine Deligne-Lusztig varieties.

\section{Affine Hecke algebras}

\subsection{} Let $\fR=(X, R, Y, R^\vee, F_0)$ be a based root datum, where $R \subset X$ is the set of roots, $R^\vee \subset Y$ is the set of coroots and $F_0 \subset R$ is the set of simple roots. By definition, there exist a bijection $\a \mapsto \a^\vee$ from $R$ to $R^\vee$ and a perfect pairing $\< , \>: X \times Y \to \ZZ$ such that $\<\a, \a^\vee\>=2$ and the corresponding reflections $s_\a: X \to X$ stabilizes $R$ and $s_{\a}^\vee: Y \to Y$ stabilizes $R^\vee$. We denote by $R^+ \subset R$ the set of positive roots determined by $F_0$. Let $X^+=\{\l \in X; \<\l, \a^\vee\> \ge 0, \, \forall \a \in R^+\}$.

The reflections $s_\a$ generate the Weyl group $W_0=W(R)$ of $R$ and $S_0=\{s_\a; \a \in F_0\}$ is the set of simple reflections.

An automorphism of $\fR$ is an automorphism $\d$ of $X$ such that $\d(F_0)=F_0$. Let $\G$ be a subgroup of automorphisms of $\fR$.

\subsection{}\label{length} Let $V=X \otimes_\ZZ \RR$. For $\a \in R$ and $k \in \ZZ$, set $$H_{\a, k}=\{v \in V; \<v, \a^\vee\>=k\}.$$ Let $\fH=\{H_{\a, k}; \a \in R, k \in \ZZ\}$.  Connected components of $V-\cup_{H \in \fH}H$ are called {\it alcoves}. Let $$C_0=\{v \in V; 0 < \<v, \a^\vee\> <1, \, \forall \a \in R^+\}$$ be the fundamental alcove.

Let $W=\ZZ R \rtimes W_0$ be the affine Weyl group and $S \supset S_0$ be the set of simple reflections in $W$. Then $(W, S)$ is a Coxeter group. Set $\tW=(X \rtimes W_0)\rtimes \G=X \rtimes (W_0 \rtimes \G)$. Then $W$ is a subgroup of $\tW$. Both $W$ and $\tW$ can be regarded as groups of affine transformations of $V$, which send alcoves to alcoves. For $\l \in X$, we denote by $t^\l \in W$ the corresponding translation. For any hyperplane $H=H_{\a,k} \in \fH$ with $\a \in R$ and $k \in \ZZ$, we denote by $s_H=t^{k \a} s_\a \in W$ the reflection of $V$ along $H$.

For any $\tw \in \tW$, we denote by $\ell(\tw)$ the number of hyperplanes in $\fH$ separating $C_0$ from $\tw(C_0)$. By \cite{IM65}, the length function is given by the following formula $$\ell(t^\chi w \t)=\sum_{\a, w \i(\a) \in R^+} |\<\chi, \a^\vee\>|+\sum_{\a\in R^+, w \i(\a)\in R^-} |\<\chi, \a^\vee\>-1|.$$ Here $\chi \in X$,$w \in W_0$ and $\t \in \G$.

If $\tw \in W$, then $\ell(\tw)$ is just the word length in the Coxeter system $(W, S)$. Let $\Om=\{\tw \in \tW; \ell(\tw)=0\}$. Then $\tW=W \rtimes \Om$.

\subsection{}\label{IMpresent} Let $q_s^{\frac{1}{2}}, s \in S$ be indeterminates. We assume that $q_s^{\frac{1}{2}}=q_t^{\frac{1}{2}}$ if $s, t$ are conjugate in $\tW$. Let $\ca=\ZZ[q_s^{\frac{1}{2}}, q_s^{-\frac{1}{2}}]_{s \in S}$ be the ring of Laurant polynomials in $q_s^{\frac{1}{2}}, s \in S$ with integer coefficients.

The (generic) Hecke algebra $\ch$ associated to the extend affine Weyl group $\tW$ is an associative $\ca$-algebra with basis $\{T_{\tw}; \tw \in \tW\}$ subject to the following relations \begin{gather*} T_{\tx} T_{\ty}=T_{\tx \ty}, \quad \text{ if } \ell(\tx)+\ell(\ty)=\ell(\tx \ty); \\ (T_s-q_s^{\frac{1}{2}})(T_s+q_s^{-\frac{1}{2}})=0, \quad \text{ for } s \in S. \end{gather*}

If $q_s^{\frac{1}{2}}=q_t^{\frac{1}{2}}$ for all $s, t \in S$, then we call $\ch$ the (generic) Hecke algebra with equal parameter.

This is the Iwahori-Matsumoto presentation of $\ch$. It reflects the structure of (quasi) Coxeter group $\tW$.

\subsection{}\label{th} In this section, we recall the Bernstein presentation of $\ch$. It is used to construct a basis of the center of $\ch$ and is useful in the study of representations of $\ch$.

For any $\l \in X$, we may write $\l$ as $\l=\chi-\chi'$ for $\chi, \chi' \in X^+$. Now set  $\th_\l=T_\chi T_{\chi'} \i$. It is easy to see that $\th_\l$ is independent of the choice of $\chi, \chi'$. The following results can be found in \cite{Lu}.

(1) $\th_\l \th_{\l'}=\th_{\l+\l'}$ for $\l, \l' \in X$.

(2) The set $\{\th_\l T_w; \l \in X, w \in W_0\}$ and $\{T_w \th_\l; \l \in X, w \in W_0\}$ are $\ca$-basis of $\ch$.

(3)  For $\l \in X^+$, set $z_\l=\sum_{\l' \in W \cdot \l} \th_{\l'}$. Then $z_{\l}, \l \in X^+$ is an $\ca$-basis of the center of $\ch$.

(4) $\th_{\chi}T_{s_\a}-T_{s_\a} \th_{s_\a(\chi)}=(q_{s_\a}^{\frac{1}{2}}-q_{s_\a}^{-\frac{1}{2}})\frac{\th_{\chi}-\th_{s_\a(\chi)}}{1-\th_{-\a}}$ for $\a \in F_0$ such that $\a^\vee \notin 2 Y$ and $\chi \in X$.

The following special cases will be used a lot in this paper.

(5) Let $\a \in F_0$ and $\chi \in X$. If $\<\chi, \a^\vee\>=0$, then $\th_\chi T_{s_\a}=T_{s_\a} \th_\chi$.

(6) Let $\a \in F_0$ and $\chi \in X$. If $\<\chi, \a^\vee\>=1$, then $\th_{s_\a(\chi)}=T_{s_\a} \i \th_\chi T_{s_\a} \i$.

\subsection{}\label{ptj} For any $J \subset S_0$, let $R_J$ be the set of roots spanned by $\a$ for $\a \in J$ and $R^\vee_J$ be the set of coroots spanned by $\a^\vee$ for $\a \in J$. Let $\fR_J=(X, R_J, Y, R^\vee_J, J)$ be the based root datum corresponding to $J$. Let $W_J \subset W_0$ be the subgroup generated by $s_\a$ for $\a \in J$ and set $\tW_J=(X \rtimes W_J) \rtimes \G_J$. Here $\G_J=\{\d \in \G; \d(R_J)=R_J\}$.
As in \S\ref{length}, we set $\fH_J=\{H_{\a, k} \in \fH; \a \in R_J, k \in \ZZ\}$ and $C_J=\{v \in V; 0 < \<v, \a^\vee\> < 1, \a \in R_J^+\}$. For any $\tw \in \tW_J$, we denote by $\ell_J(\tw)$ the number of hyperplanes in $\fH_J$ separating $C_J$ from $\tw C_J$.

We denote by $\tW^J$ (resp. ${}^J \tW$) the set of minimal coset representatives in $\tW/W_J$ (resp. $W_J \backslash \tW$). For $J, K \subset S_0$, we simply write $\tW^J \cap {}^K \tW$ as ${}^K \tW^J$.


Let $\ch_J \subset \ch$ be the subalgebra generated by $\th_\l$ for $\l \in X$ and $T_w$ for $w \in W_J \rtimes \G_J$. We call $\ch_J$ a parabolic subalgebra of $\ch$.

It is known that $\ch_J$ is the Hecke algebra associated to the extend affine Weyl group $\tW_J$ and the parameter function $p^{\frac{1}{2}}_t$, where $t$ ranges over simple reflections in $\tW_J$. The parameter function $p^{\frac{1}{2}}_t$ is determined by $q^{\frac{1}{2}}_s$ (see \cite[1.2]{OS}). We denote by $\{T^J_{\tw}\}_{\tw \in \tW_J}$ the Iwahori-Matsumoto basis of $\ch_J$.

\section{The Iwahori-Matsumoto presentation of $\bar \ch$}

\subsection{}\label{to} We follow \cite{HN2}.

For $w, w' \in \tW$ and $s \in S$, we write $w \xrightarrow{s} w'$ if $w'=s w s$ and $\ell(w') \le \ell(w)$.  We write $w \to w'$ if there is a sequence $w=w_0, w_1, \cdots, w_n=w'$ of elements in $\tW$ such that for any $k$, $w_{k-1} \xrightarrow{s} w_k$ for some $s \in S$.

We write $w \approx w'$ if $w \to w'$ and $w' \to w$. It is easy to see that $w \approx w'$ if $w \to w'$ and $\ell(w)=\ell(w')$.

We call $\tw, \tw' \in \tW$ {\it elementarily strongly conjugate} if $\ell(\tw)=\ell(\tw')$ and there exists $x \in W$ such that $\tw'=x \tw x \i$ and $\ell(x \tw)=\ell(x)+\ell(\tw)$ or $\ell(\tw x \i)=\ell(x)+\ell(\tw)$. We call $\tw, \tw'$ {\it strongly conjugate} if there is a sequence $\tw=\tw_0, \tw_1, \cdots, \tw_n=\tw'$ such that for each $i$, $\tw_{i-1}$ is elementarily strongly conjugate to $\tw_i$. We write $\tw \sim \tw'$ if $\tw$ and $\tw'$ are strongly conjugate. We write $\tw \tilde \sim \tw'$ if $\tw \sim \d \tw'\d^{-1}$ for some $\d\in\Om$.

Now we recall one of the main results in \cite{HN2}.

\begin{thm}\label{min}
Let $\co$ be a conjugacy class of $\tW$ and $\co_{\min}$ be the set of minimal length elements in $\co$. Then

(1) For any $\tw' \in \co$, there exists $\tw'' \in \co_{\min}$ such that $\tw' \rightarrow \tw''$.

(2) Let $\tw', \tw'' \in \co_{\min}$, then $\tw' \tilde \sim \tw''$.
\end{thm}

\subsection{}\label{min-tw} Let $h, h' \in \ch$, we call $[h, h']=h h'-h'h$ the {\it commutator} of $h$ and $h'$. Let $[\ch, \ch]$ be the $\ca$-submodule of $\ch$ generated by all commutators. We call the quotient $\ch / [\ch, \ch]$ the {\it cocenter} of $\ch$ and denote it by $\bar \ch$.

It follows easily from definition that $T_{\tw}\equiv T_{\tw'} \mod [\ch, \ch]$ if $\tw \tilde \sim \tw'$. Hence by Theorem \ref{min} (2), for any conjugacy class $\co$ of $\tW$ and $\tw, \tw' \in \co_{\min}$, $T_{\tw}\equiv T_{\tw'} \mod [\ch, \ch]$. We denote by $T_\co$ the image of $T_{\tw}$ in $\bar \ch$ for any $\tw \in \co_{\min}$.

\begin{thm}\label{IMbasis}
(1) The elements $\{T_\co\}$, where $\co$ ranges over all the conjugacy classes of $\tW$, span $\bar \ch$ as an $\ca$-module.

(2) If $q_s^{\frac{1}{2}}=q_t^{\frac{1}{2}}$ for all $s, t \in S$, then $\{T_\co\}$ is a basis of $\bar \ch$.
\end{thm}

We call $\{T_\co\}$ the Iwahori-Matsumoto presentation of the cocenter $\bar \ch$ of affine Hecke algebra $\ch$.

The equal parameter case was proved in \cite[Theorem 5.3 \& Theorem 6.7]{HN2}. Part (1) for the unequal parameter case can be proved in the same way as in loc. cit. We expect that Part (2) remains valid for unequal parameter case. One possible approach is to use the classification of irreducible representations and a generalization of density theorem. We do not go into details in this paper.

\section{Some length formulas}

\subsection{} The strategy to prove Theorem B in this paper is as follows. For a given conjugacy class $\co$, we

\begin{itemize}

\item construct a minimal length element in $\co$, which is used for the Iwahori-Matsumoto presentation of $\bar \ch$;

\item  construct a suitable $J$, and an element in $\co \cap \tW_J$, of minimal length in its $\tW_J$-conjugacy class, which is used for the Bernstein presentation of $\bar \ch$;

\item find the explicit relation between the two different elements.
\end{itemize}

To do this, we need to relate the length function on $\tW$ with the length function on $\tW_J$ for some $J \subset S_0$. This is what we will do in this section. Another important technique is the ``partial conjugation'' method introduced in \cite{H1}, which will be discussed in the next section.

\subsection{} 

Let $n=\sharp(W_0 \rtimes \G)$. For any $\tw \in \tW$, $\tw^n=t^\l$ for some $\l \in X$. We set $\nu_{\tw}=\l/n \in V$ and call it the {\it Newton point of} $\tw$. Let $\bar \nu_{\tw}$ be the unique dominant element in the $W_0$-orbit of $\nu_{\tw}$. Then the map $\tW \to V, \tw \mapsto \bar \nu_{\tw}$ is constant on the conjugacy class of $\tW$. For any conjugacy class $\co$, we set $\nu_\co=\bar \nu_{\tw}$ for any $\tw \in \co$ and call it the {\it Newton point of} $\co$.

For $\tw \in \tW$, set \[V_{\tw}=\{v \in V; \tw(v)=v+\nu_{\tw}\}.\] By \cite[Lemma 2.2]{HN2}, $V_{\tw} \subset V$ is a nonempty affine subspace and $\tw V_{\tw}=V_{\tw}+\nu_{\tw}=V_{\tw}$. Let $p: \tW=X \rtimes (W_0 \rtimes \G) \to W_0 \rtimes \G$ be the projection map. Let $u$ be an element in $V_{\tw}$. By the definition of $V_{\tw}$, $V^{p(\tw)}=\{v-u; v \in V_{\tw}\}$. In particular, $\nu_{\tw} \in V^{p(\tw)}$.


Let $E \subset V$ be a convex subset. Set $\fH_E=\{H \in \fH; E \subset H\}$ and $W_E \subset W$ to be the subgroup generated by $s_H$ with $H \in \fH_E$. We say a point $p \in E$ is {\it regular} in $E$ if for any $H \in \fH$, $v \in H$ implies that $E \subset H$. Then regular points of $E$ form an open dense subset of $V_{\tw}$.

For any $\l \in V$, set $J_\l=\{s \in S_0; s(\l)=\l\}$.

\begin{prop}\label{mininW}
Let $\tw \in \tW$ such that $\bar C_0$ contains a regular point $e$ of $V_{\tw}$. Then $\tw$ is of minimal length in its conjugacy class if and only if it is of minimal length in its $W_{V_{\tw}}$-conjugacy class;
\end{prop}
\begin{proof}
Note that for any $x \in W_{V_{\tw}}$, $\bar C_0$ contains a regular point of $V_{\tw}=x\i V_{\tw}=V_{x\i \tw x}$, hence by \cite[Proposition 2.5 \& Proposition 2.8]{HN2}, the minimal length of elements in the conjugacy class of $\tw$ equals \begin{align*} \<\bar \nu_{\tw}, \rho^\vee\>+\min_C \sharp \fH_{V_{\tw}}(C, \tw C) &=\<\bar \nu_{\tw}, \rho^\vee\>+\min_{x \in W_{V_{\tw}}} \sharp \fH_{V_{\tw}}(x C_0, \tw x C_0)\\ &=\<\bar \nu_{\tw}, \rho^\vee\>+\min_{x \in W_{V_{\tw}}} \sharp \fH_{V_{\tw}}(C_0, x\i \tw x C_0) \\ &=\min_{x \in W_{V_{\tw}}} \ell(x\i \tw x), \end{align*} where $C$ ranges over all connected components of $V-\cup_{H \in \fH_{\tw}}H$ and $\rho^\vee=\frac{1}{2}\sum_{\a \in R^+}\a^\vee$.
\end{proof}

\begin{prop}\label{fact'}
Let $\tw \in \tW$ such that $\bar C_0$ contains a regular point of $V_{\tw}$. Let $J \subset S_0$. Assume there exists $z \in {}^J W_0$ such that $z \tw z\i \in \tW_J$. Then $$\ell(\tw)=\ell_J(z \tw z\i)+\<\bar \nu_{\tw}, 2 \rho^\vee\>-\<\bar \nu_{z \tw z\i}^J, 2 \rho_J^\vee\>,$$ where $\bar \nu_{z \tw z\i}^J$ denotes the unique $J$-dominant element in the $W_J$-orbit of $\nu_{z \tw z\i}$ and $\rho_J^\vee=\frac{1}{2}\sum_{\a \in R_J^+}\a^\vee$. In particular, if $J \subset J_{\nu_{z \tw z\i}}$, we have $$\ell(\tw)=\ell_J(z \tw z\i)+\<\bar \nu_{\tw}, 2 \rho^\vee\>.$$
\end{prop}

\begin{proof}
By \cite[Proposition 2.8]{HN2} we have $$\ell(\tw)=\<\bar \nu_{\tw}, 2\rho^\vee\>+\sharp\fH_{V_{\tw}}(C, \tw C),$$ where $C$ is the connected component of $V-\cup_{H \in \fH_{V_{\tw}}}H$ containing $C_0$. Since $z \in {}^J W_0$, $z C_0 \subset C_J$ and hence $\bar C_J$ contains a regular point of $z V_{\tw}=V_{z \tw z\i}$. Applying \cite[Proposition 2.8]{HN2} to $\ell_J$ instead of $\ell$ we obtain \begin{align*}\ell_J(z \tw z\i)&=\<\bar \nu_{\tw}^J, 2\rho_J^\vee\>+\sharp \bigl(\fH_J \cap \fH_{V_{z \tw z\i}}(C', z \tw z\i C') \bigr) \\ &=\<\bar \nu_{\tw}^J, 2\rho_J^\vee\>+\sharp \fH_{V_{z \tw z\i}}(C', z \tw z\i C'),\end{align*} where $C'$ is the connected component of $V-\cup_{H \in \fH_{V_{z \tw z\i}}}H$ containing $C_J$ and the second equality follows from the fact that $\fH_{V_{z \tw z\i}} \subset \fH_J$. Since $z C = C'$, the map $H \mapsto z H$ induces a bijection between $\fH_{V_{\tw}}(C, \tw C)$ and $\fH_{V_{z \tw z\i}}(C', z \tw z\i C')$.
\end{proof}

\begin{lem}\label{zs}
Let $J \subset S_0$ and $z \in {}^J W_0$. Let $s \in S$ and $t=z s z \i$.

(1) If $t \in \tW_J$, then $\ell_J(t)=1$.

(2) If $t \notin \tW_J$, then $z s=x z'$ for some $x \in \tW_J$ with $\ell_J(x)=0$ and $z' \in {}^J W_0$.
\end{lem}
\begin{proof}
Assume $s=s_H$ is the reflection along some hyperplane $H \in \fH$. Since $s \in S$, $\bar C_0$ contains some regular point of $H$. Since $z \in {}^J W_0$, $z C_0 \subset C_J$. If $t \in \tW_J$, then $\bar C_J$ contains some regular point of $H'=z H$ and hence $t=s_{H'}$ is of length one with respect to $\ell_J$.

If $t \in W_0-W_J$, then $s \in S_0$ and $z s \in {}^J W_0$. In this case, $z'=z s$ and $x=1$. If $t \notin \tW_J \cup W_0$, then $s=t^{\th}s_{\th}$ for some maximal coroot $\th^\vee$ with $z(\th) \notin R_J$. Then $zs=t^{z(\th)} u z'$ for some $u \in W_J$ and $z' \in {}^J W_0$. Let $\a \in R_J^+$. Since $z', z \in {}^J W_0$ and that $\th^\vee$ is a maximal coroot, we have $$\<z(\th), \a^\vee\>=\left\{
                 \begin{array}{ll}
                   1, & \hbox{if $u\i(\a)<0$;} \\
                   0, & \hbox{Otherwise.}
                 \end{array}
               \right.
$$
In other words, $\ell_J(t^{z(\th)} u)=0$.
\end{proof}

\begin{cor}\label{cor4.4}
Let $\tw' \in \tW$ and $z' \in {}^J W_0$ such that $z' \tw {z'}\i \in \tW_J$. Let $s \in S$ such that $\tw'$ and $\tw=s \tw' s$ are of the same length. Let $z$ be the unique minimal element of the coset $W_J z' p(s)$. Then $z \tw z\i$ and $z' \tw' {z'}\i$ belong to the same $\tW_J$-conjugacy class and $$\ell_J(z \tw z\i)=\ell_J(z' \tw' {z'}\i).$$
\end{cor}

\begin{proof}
The first statement follow form the construction of $z$.

Without loss of generality, we may assume that $\tw' s >  \tw' > s \tw'$. Let $t=z' s {z'}\i$. If $t \in \tW_J$, then $\ell_J(t)=1$ by Lemma \ref{zs}. Since $\tw' > s \tw'$, the reflection hyperplane $H \in \fH$ of $s$ separates $C_0$ from $\tw' C_0$. Hence $z' H$ separates $C_J$ from $z' \tw' {z'}\i C_J$ since $z' C_0 \subset C_J$, which means that $z' \tw' {z'}\i > t z' \tw' {z'}\i$. Similarly, $z' \tw' {z'}\i t > z' \tw' {z'}\i$. Therefore $\ell_J(z \tw z\i)=\ell_J(t z' \tw' {z'}\i t)=\ell_J(z' \tw' {z'}\i)$.

If $t \notin \tW_J$, then $z s=x \i z'$ for some $x \in \tW_J$ with $\ell_J(x)=0$. Hence $z \tw z\i=x \i z' \tw' {z'}\i x$ and $\ell_J(z \tw z\i)=\ell_J(z' \tw' {z'}\i)$.
\end{proof}

\begin{prop} \label{mm}
Let $\co$ be a conjugacy class of $\tW$ and $J \subset S_0$ such that $\co \cap \tW_J \neq \emptyset$. Let $\tw \in \co_{\min}$ and $z \in {}^J W_0$, such that $z \tw z\i \in \tW_J$. Then $z \tw z\i$ is of minimal length (with respect to $\ell_J$) in its $\tW_J$-conjugayc class.
\end{prop}
\begin{proof}
By \cite[Proposition 2.5 \& Lemma 2.7]{HN2}, there exists $\tw \to \tw' \in \co_{\min}$ such that $\bar C_0$ contains a regular point of $V_{\tw'}$. By Corollary \ref{cor4.4}, it suffices to consider the case that $\bar C_0$ contains a regular point of $V_{\tw}$. By Proposition \ref{mininW} and Proposition \ref{fact'}, $$\ell_J(z \tw z\i)=\min_{x \in W_{V_{\tw}}} \ell_J(z x \tw x\i z\i)=\min_{y \in W_{V_{z \tw z\i}}} \ell_J(y z \tw z\i y\i).$$ Note that $\bar C_J$ contains a regular point of $V_{z \tw z\i}$. Applying Proposition \ref{mininW} to $\ell_J$ and $z \tw z\i$ we obtain the desired result.
\end{proof}

\section{A family of partial conjugacy classes}

\subsection{} In this section, we consider an arbitrary Coxeter group $(W, S)$. Let $T=\cup_{w \in W} w S w \i \subset W$ be the set of reflections in $W$. Let $R=\{\pm 1\} \times T$. For $s \in S$, define $U_s: R \to R$ by $U_s(\e, t)=(\e (-1)^{\d_{s, t}}, s t s)$.

Let $\G$ be a subgroup of automorphisms of $W$ such that $\d(S)=S$ for all $\d \in \G$. Let $\tW=W \rtimes \G$. For any $\d \in \G$, define $U_\d: R \to R$ by $U_\d(\e, t)=(\e, \d(t))$. Then $U_\d U_s U_{\d \i}=U_{\d(s)}$ for $s \in S$ and $\d \in \G$.

We have the following result.

\begin{prop}
(1) There is a unique homomorphism $U$ of $\tW$ into the group of permutations of $R$ such that $U(s)=U_s$ for all $s \in S$ and $U(\d)=U_\d$ for all $\d \in \G$.

(2) For any $w \in \tW$ and $t \in T$, $t w<w$ if and only if for $\e=\pm 1$, $U(w \i)(\e, t)=(-\e, w \i t w)$.
\end{prop}

The case $\G=\{1\}$ is in \cite[Proposition 1.5 \& Lemma 2.2]{L99}. The general case can be reduced to that case easily.

\subsection{} Let $J \subset S$. We consider the action of $W_J$ on $\tW$ by $w \cdot w'=w w' w \i$ for $w \in W_J$ and $w \in \tW$. Each orbit is called a {\it $W_J$-conjugacy class} or a {\it partial conjugacy class} of $\tW$ (with respect to $W_J$). We set $\G_J=\{\d \in \G; \d(J)=J\}$.

\begin{lem}\label{min-in-wj}
Let $I \subset S$ and $w \in W_I \rtimes \G_I$. Then $w$ is of minimal length in its $W_I$-conjugacy class if and only if $w$ is of minimal in its $W$-conjugacy class.
\end{lem}

\begin{proof}
The ``if'' part is trivial.

Now we show the ``only if'' part. Suppose that $w$ is a minimal length element in its $W_I$-conjugacy class. An element in the $W$-conjugacy class of $w$ is of the form $x w x\i$ for some $x \in W$. Write $x=x_1 y$, where $x_1 \in W^I$ and $y \in W_I$. Then $x w x \i=x_1 (y w y \i) x_1 \i$. Here $y w y \i \in W_I$ is in the $W_I$-conjugacy class of $w$. Hence $\ell(y w y \i) \ge \ell(y)$. Now \begin{align*} \ell(x w x \i) & \ge \ell(x_1 (y w y \i))-\ell(x_1)=\ell(x_1)+\ell(y w y \i)-\ell(x_1) \\ &=\ell(y w y \i) \ge \ell(w).\end{align*} Thus  $w$ is a minimal length element in its $W$-conjugacy class.
\end{proof}

\subsection{} In general, a $W_J$-conjugacy class in $\tW$ may not contain any element in $W_J \rtimes \G_J$. To study the minimal length elements in this partial conjugacy class, we introduce the notation $I(J, w)$.

For any $w \in {}^J \tW$, set $$I(J, w)=\max\{K \subset J; w K w\i=K\}.$$ Since $w(K_1 \cup K_2)w\i=w K_1 w\i \cup w K_2 w\i$, $I(J, w)$ is well-defined. We have that

(a)  $I(J, w)=\cap_{i \ge 0} w^{-i} J w^i$.

Set $K=\cap_{i \ge 0} w^{-i} J w^i$. Let $s \in I(J, w)$. Then $w^i s w^{-i} \in I(J, w) \subset J$ for all $i$. Thus $s \in K$. On the other hand, $w K w\i \subset K$. Since $K$ is a finite set, $w K w\i=K$. Thus $K \subset I(J, w)$.

(a) is proved.

\begin{lem}\label{infty}
Let $w \in {}^J \tW$ and $x \in W_J$. Then $x \in W_{I(J, w)}$ if and only if $w^{-i} x w^i \in W_J$ for all $i \in \ZZ$.
\end{lem}

\begin{proof}
If $x \in W_{I(J, w)}$, then $w \i x w \in W_{w\i(I(J, w))}=W_{I(J, w)}$. So $w^{-i} x w^i \in W_J$ for all $i \in \ZZ$.

Suppose that $w^{-i} x w^i \in W_J$ for all $i \in \ZZ$. We write $x$ as $x=u x_1$, where $u \in W_{J \cap w \i J w}$ and $x_1 \in {}^{J \cap w \i J w} W$. By \cite[2.1(a)]{L4}, $w x_1 \in {}^J W$ and $w x \in W_J (w x_1)$. Since $w x w \i \in W_J$, $w x \in W_J w$. Therefore $W_J (w x_1) \cap W_J w \neq \emptyset$ and $w x_1=w$. So $x_1=1$ and $x \in W_{J \cap w \i J w}$.

Applying the same argument, $x \in W_{\cap_{i \ge 0} w^{-i} J w^i}=W_{I(J, w)}$.
\end{proof}

\subsection{} Similar to $\S$\ref{to}, for $w, w' \in \tW$, we write $w \to_J w'$ if there is a sequence $w=w_0, w_1, \cdots, w_n=w'$ of elements in $\tW$ such that for any $k$, $w_{k-1} \xrightarrow{s} w_k$ for some $s \in J$. The notations $\sim_J$ and $\approx_J$ are defined in a similar way.

The following result is proved in \cite{H1}.

\begin{thm}\label{partial}
Let $\co$ be a $W_J$-conjugacy class of $\tW$. Then there exists a unique element $\tw \in {}^J \tW$ and a $W_{I(J, \tw)}$-conjugacy class $C$ of $W_{I(J, \tw)} \tw$ such that $\co \cap W_{I(J, \tw)} \tw=C$. In this case,

(1) for any $\tw' \in \co$, there exists $x \in W_{I(J, \tw)}$ such that $\tw' \to_J x \tw$.

(2) for any two minimal length elements $x, x'$ of $C$, $x \sim_{I(J, \tw)} x'$.
\end{thm}

Now we prove the following result.

\begin{thm}\label{fact}
Let $I \subset J \subset S$ and let $w, u \in \tW$ be of minimal length in the same $W_J$-conjugacy class such that $u \in {}^J \tW$, $w \in {}^I \tW$ and $w I w\i=I$. Then there exists $h \in {}^{I(J,u)} W_J {}^I$ such that $h I h\i \subset I(J,u)$ and $h w h\i=u$.
\end{thm}

\begin{rmk}
This result can be interpreted as a conjugation on a family of partial conjugacy classes in the following sense. Let $\ci_1$ be the set of $W_I$-conjugacy classes that intersects $W_I w$ and $\ci_2$ be the set of $W_J$-conjugacy classes that intersects $W_{I(J, u)} u$. Then

(1) There is an injective map $\ci_1 \to \ci_2$ which sends a $W_I$-conjugacy class $\co_1$ in $\ci_1$ to the unique $W_J$-conjugacy class $\co_2$ in $\ci_2$ that contains $\co_1$.

(2) Conjugating by $h$ sends $\co_1 \cap W_I w$ into $\co_2 \cap W_{I(J, u)} u$.
\end{rmk}

\begin{proof} Let $b=f \e \in W_J$ with $f \in  W_J {}^I$ and $\e \in W_I$ such that $b w b\i=u$. We show that

(a) $f w f\i=u$.

Set $x=\e w\e\i w \i$. Then $x \in W_I$ and $u=f x w f\i$.  Suppose that $x \neq 1$. Then $s x<x$ for some $s \in I$. Set $t=f s f \i$. Since $f \in W^I$, $t f=f s>f$. Thus $U(f \i)(\e, t)=(\e, f \i t f)=(\e, s)$. Since $s x<x$, $U(x\i)(\e, s)=(-\e, x \i s x)$. Notice that $w \in {}^I W$ with $w I w\i=I$ and $f \i \in {}^I W$. Thus $w f \i \in {}^I W$. Hence $U(f w \i)(-\e, x \i s x)=(-\e, f w \i x \i s x w f \i)$.

Therefore $U(f w\i x \i f \i)(\e, t)=(-\e, f w \i x \i s x w f \i)$ and $$t f x w f \i=f (s x) w f \i<f x w f \i=u.$$

Applying this argument successively, we have $f w f\i<u.$ This contradicts our assumption that $u$ is of minimal length in the $W_K$-conjugacy class containing $f w f\i$. Hence $x=1$ and $f w f\i=u$.

(a) is proved.


Now we write $f$ as $f=\pi h$ with $\pi \in W_{I(J,u)}$ and $h \in {}^{I(J,u)} W_J {}^I$. Then $w=h \i (\pi \i u \pi u \i) u h$. Similar to the proof of (a), we have that $w \ge h \i u h$. By our assumption, $w$ is a minimal length element in the $W_J$-conjugacy class of $h \i u h$. Thus $w=h \i u h$.

For any $x \in W_I$ and $i \in \ZZ$, $h w^i=u^i h$ and $$u^i (h x h \i) u^{-i}=(h w^i) x (h w^i) \i=h (w^i x w^{-i}) h \i \in h W_I h \i \subset W_J.$$ By Lemma \ref{infty}, $h x h \i \in W_{I(J, u)}$. Thus $h W_I h \i \subset W_{I(J, u)}$. Since $h \in {}^{I(J,u)} W_J {}^I$, we have $h I h\i \subset I(J, u)$.
\end{proof}

\section{Bernstein presentation of the cocenter of $\ch$}

\subsection{}\label{def-tw0}
We fix a conjugacy class $\co$ of $\tW$ and will construct a subset $J$ of $S_0$, as small as possible, such that $T_{w_\co} \in \ch_J+[\ch, \ch]$.

By \cite[Proposition 2.5 \& Lemma 2.7]{HN2}, there exists $\tw' \in \co_{\min}$ such that $\bar C_0$ contains a regular point $e'$ of $V_{\tw'}$. We choose $v \in V$ such that $V_{\tw'}=V_{\tw'}+v$ and $v, \nu_{\tw'} \in \bar C$ for some Weyl Chamber $C$. We write $J$ for $J_{\nu_\co} \cap J_{\bar v}$. Let $z \in {}^J W_0$ with $z(\nu_{\tw'})=\nu_\co$ and $z(v)=\bar v$. Set $\tw_0=z \tw' z \i$. By Proposition \ref{mm}, $\tw_0$ is of minimal length (with respect to $\ell_J$) in its $\tW_J$-conjugacy class.

Unless otherwise stated, we keep the notations in the rest of this section. The main result of this section is

\begin{thm}\label{thm4.1}
We keep the notations in $\S$\ref{def-tw0}. Then
$$T_{w_\co} \equiv T^{J}_{\tw_0} \mod [\ch, \ch].$$
\end{thm}



\subsection{} The idea of the proof is as follows.

Suppose $\tw_0=t^{\l_0} w_0$ and $\tw'=t^{\l'} w'$. Then we need to compare $T^{J}_{t^{\l_0}}$ and $T_{t^{\l'}}$. Although $\l_0$ and $\l$ are in the same $W_0$-orbit, the relate between $T^{J}_{t^{\l_0}}$ and $T_{t^{\l'}}$ is complicated. Roughly speaking, we write $\l_0$ as $\l_0=\mu_1-\mu_2$ for $J$-dominant coweights, i.e., $\<\mu_1, \a_i\>, \<\mu_2, \a_i\> >0$ for $i \in J$.  Then $$T^{J}_{t^{\l_0}}=T^{J}_{t^{\mu_1}} (T^{J}_{t^{\mu_2}}) \i=\th_{\mu_1} \th_{\mu_2} \i.$$ The right hand side is not easy to compute.

To overcome the difficulty, we replace $\tw_0$ by another minimal length element $\tw_1$ in its $\tW_{J}$-conjugacy class whose translation part is $J$-dominant and replace $\tw'$ by another minimal length element $\tw_2$ in $\co$ and study the relation between $\tw_1$ and $\tw_2$ instead. The construction of $\tw_1$ and $\tw_2$ uses ``partial conjugation action''.

\subsection{}\label{def-tw1} Recall that $e' \in \bar C_0$ is a regular element of $V_{\tw'}$. Set $e=z(e')$.

Since $e \in z(\bar C_0)$ and is a regular point of $V_{\tw_0}$, we have

(1) $0 \le |\<e, \a^\vee\>| \le 1$ for any $\a \in R$.

(2) $\<e, \a^\vee\> \ge 0$ for any $\a \in R_{J}^+$.

(3) If $e \in H_{\a, k}$ for some $\a \in R$ and $k \in \ZZ$, then $V_{\tw_0} \subset H_{\a, k}$ and $\a \in R_{J}$. In particular, $J_e \subset J$ and by (2), $R_{J_e}=\{\a \in R; \<e, \a^\vee\>=0\}$.

By Theorem \ref{partial}, the $W_{J_e}$-conjugacy class of $\tw_0$ contains a minimal length element $\tw_1$ of the form $\tw_1=t^{\l}w_1 x_1$ with $\l \in X$, $w_1 \in W_{J} \rtimes \G_{J}$ and $x_1 \in W_{I(J_e, t^\l w_1)}$ such that $t^{\l}w_1 \in {}^{J_e} \tW$ and $x_1$ is of minimal length in its $\Ad(w_1)$-twisted conjugacy class of $W_{I(J_e, t^\l w_1)}$.

By Theorem \ref{partial}, there exists a minimal length element $\tw_1$ in the $W_0$-conjugacy of $\tw_0$ and $\tw'$, which has the form $\tw_2=t^{\bar \l} w_2 x_2$ such that $t^{\bar \l} w_2 \in {}^{S_0} \tW$ and $x_2 \in W_{I(S_0, t^{\bar \l} w_2)}$. Since $\tw' \in \co_{\min}$, $\tw_1 \in \co_{\min}$.

We have the following results on $\l$ and $e$ constructed above.

\begin{lem}\label{ineq'}
Keep notations in \S \ref{def-tw1}. Then we have

(1) For any $\a \in R^+$, $\<\tw_1(e), \a^\vee\> > -1$.

(2) $\<\l, \a^\vee\> \ge -1$ for any $\a \in R^+$.

(3) $\<\l, \a^\vee\> \ge 0$ for any $\a \in R_{J}^+$.

(4) For any $\a \in R^+$, if $\<\l, \a^\vee\> = -1$, then $\<e, \a^\vee\> < 0$.
\end{lem}
\begin{proof}
For any $\a \in R^+$, \[\<\l, \a^\vee\>+\<w_1(e), \a^\vee\>=\<\l+w_1 x_1(e), \a^\vee\>=\<\tw_1(e), \a^\vee\>=\<e+\nu_\co,\a^\vee\>.\]

(1) By $\S$\ref{def-tw1} (1), $\<e, \a^\vee\> \ge -1$. So $\<e+\nu_\co, \a^\vee\> \ge -1$. If $\<e+\nu_\co, \a^\vee\>=-1$, then $\<e, \a^\vee\>=-1$. Therefore $\a \in R_{J}^+$ by \S \ref{def-tw1} (3), which contradicts $\S$\ref{def-tw1} (2).

(2) By $\S$\ref{def-tw1} (1), $\<w_1(e), \a^\vee\> \le 1$. By (1), $\<\l, \a^\vee\>=\<\tw_1(e),\a^\vee\>-\<w_1(e), \a^\vee\> > -2$. Since $\<\l, \a^\vee\> \in \ZZ$, $\<\l, \a^\vee\> \ge -1$.

(3) By $\S$\ref{def-tw1} (2), $0 \le \<e, \a^\vee\>=\<\nu_\co+e, \a^\vee\>=\<\l, \a^\vee\>+\<w_1(e), \a^\vee\>$. If $\<\l, \a^\vee\> < 0$, then by $\S$\ref{def-tw1} (1), $\<e, \a^\vee\>=0$ and hence $\a \in R_{J_e}^+$ by $\S$\ref{def-tw1} (4). Since $t^\l w_1 \in {}^{J_e} \tW$, $\<\l, \a^\vee\> \ge 0$, which is a contradiction. Therefore $\<\l, \a^\vee\> \ge 0$ for all $\a \in R_{J}^+$.

(4) Suppose that $\<e, \a^\vee\> \ge 0$. Since $\<\l, \a^\vee\>=-1$ and $\<w_1(e), \a^\vee\> \le 1$, $\<e+\nu_\co, \a^\vee\> \le 0$. Thus $\<e, \a^\vee\>=\<\nu_\co, \a^\vee\>=0$. Therefore $\a \in R^+_{J}$ by \S \ref{def-tw1} (3), which contradicts (3).
\end{proof}

\begin{lem}\label{fact''}
We keep the notations in $\S$\ref{def-tw1}. Then

(1) $\tw_1 \in \tW_{J}$ is of minimal length (with respect to $\ell_{J}$) in its $\tW_{J}$-conjugacy class.

(2) $t^{\l} w_1 \in {}^{J} \tW$.
\end{lem}

\begin{proof}
Since $W_{J_e} \subset W_J$ fixes $V_{\tw_0}$, we have $\tw_1 \in \tW_J$ and $V_{\tw_0}=V_{\tw_1}$.

(1) Since $\ell(\tw_1) \leqslant \ell(\tw_0)$, we have $\ell_{J}(\tw_1) \leqslant \ell_{J}(\tw_0)$ by Proposition \ref{fact'}. By Proposition \ref{mm}, $\tw_0$ is a minimal length (with respect to $\ell_{J}$) in its conjugacy class of $\tW_{J}$. So is $\tw_1$.

(2) It suffices to show that $\<\l, \a^\vee\> \ge 1$ for any $\a \in R_{J}^+$ with $w_1\i(\a)<0$. Suppose that $\<\l, \a^\vee\><1$. By Lemma \ref{ineq'} (3), $\<\l, \a^\vee\>=0$. Hence by \S \ref{def-tw1} (2), $$0 \ge \<e, w_1\i(\a^\vee)\>=\<w_1(e), \a^\vee\>=\<\tw_1(e),\a^\vee\>=\<e+\nu_\co, \a^\vee\> \ge 0.$$ Thus by \S \ref{def-tw1} (2) again, $\<e, \a^\vee\>=0$ and $\a \in R_{J_e}^+$. However, $t^{\l}w_1 \in {}^{J_e} \tW$ by our construction. Hence $\<\l, \a^\vee\> \ge 1$, which is a contradiction.
\end{proof}


Combining Proposition \ref{fact'} with Lemma \ref{fact''}, we obtain

\begin{cor}\label{formula}
Keep notations in \S \ref{def-tw1}. Then \begin{align*}\ell(z\i t^\l w_1 z)&=\<\nu_\co, 2\rho^\vee\>+\<\l, 2\rho_{J}^\vee\>-\ell(w_1),\\ \ell(z\i t^\l w_1 x_1 z)&=\<\nu_\co, 2\rho^\vee\>+\<\l, 2\rho_{J}^\vee\>-\ell(w_1)+\ell(x_1).\end{align*}
\end{cor}




\begin{lem}\label{property-y}
Keep the notations in $\S$\ref{def-tw1}. Let $y \in {}^{J_{\bar \l}} W_0$ be the unique element such that $y(\l)=\bar \l$. Then

(1) $\ell(y w_1 y\i)=2\ell(y)+\ell(w_1)$.

(2) $\<\bar \l, 2\rho^\vee\>=\<\nu_\co, 2\rho^\vee\>+\<\l, 2\rho^\vee_{J}\>+2\ell(y)$.

(3) $y J_\l y\i \subset J_{\bar \l}$.
\end{lem}
\begin{proof}

By definition, for any $\a \in R^+$, $y(\a) \in R^+$ if and only if $\<\l, \a^\vee\> \ge 0$. By Lemma \ref{ineq'} (2), $\ell(y)=\sharp\{\a \in R^+; \<\l, \a^\vee\>=-1\}$.

(1) Let $\a \in R^+$ such that $w_1\i(\a)<0$. Then $\a \in R_{J}^+$ since $w_1 \in W_{J} \rtimes \G_{J}$. Hence $\<\l, \a^\vee\> \geqslant 0$ by Lemma \ref{ineq'} (3). Therefore $y(\a)>0$. Hence $\ell(yw_1)=\ell(y)+\ell(w_1)$.

To show $\ell(y w_1 y\i)=\ell(yw_1)+\ell(y\i)$, we have to prove that for any $\b \in R^+$ with $y(\b) <0$, we have $yw_1(\b) \in R^+$.

Assume $y(\b) <0$. Then $\<\l, \b^\vee\><0$. Thus $\<\l, \b^\vee\>=-1$ by Lemma \ref{ineq'}(2). Moreover, we have $\b \notin R_{J}^+$ and $\<e, \b^\vee\> < 0$ by Lemma \ref{ineq'} (3) and (4). Since $w_1 \in W_{J} \rtimes \G_{J}$, $w_1(\b)>0$. By Lemma \ref{ineq'} (1), \begin{align*} -1 &< \<\tw_1(e), w(\b^\vee)\>=\<\l, w_1(\b^\vee)\>+\<w_1(e), w_1(\b^\vee)\>=\<\l, w_1(\b^\vee)\>+\<e, \b^\vee\>\\ & < \<\l, w_1(\b^\vee)\>.\end{align*}  Therefore $\<\l, w_1(\b^\vee)\> \ge 0$ and $y w_1(\b) \in R^+$.

(2) By Lemma \ref{min} and Lemma \ref{ineq'}(2), we have that
\begin{align*} \<\bar \l, 2\rho^\vee\>&=\sum_{\a \in R^+}|\<\l, \a^\vee\>|=\sum_{\a \in R^+}\<\l, \a^\vee\>+2\sharp\{\a \in R^+; \<\l, \a^\vee\>=-1\}\\ &=\<\l, 2\rho^\vee\>+2\ell(y). \end{align*}
Since $w_1 \in W_{J} \rtimes \G_{J}$, $w_1^k(\rho^\vee-\rho_{J}^\vee)=\rho^\vee-\rho^\vee_{J}$ for all $i \in \ZZ$. Let $m=|W_{J} \rtimes \G_{J}|$. Then $\sum_{k=1}^m w_1^k(\l)=m\nu_\co$ and \begin{align*} & \<\bar \l, 2\rho^\vee\>-\<\l, 2\rho^\vee_{J}\>=2\ell(y)+\<\l, 2(\rho^\vee-\rho^\vee_{J})\>\\ &=2\ell(y)+\frac{1}{m}\sum_{k=1}^m\<\l,2w_1^{-k}(\rho^\vee-\rho^\vee_{J})\>=2\ell(y)+\frac{1}{m}\sum_{k=1}^m\<w_1^k(\l),2(\rho^\vee-\rho^\vee_{J})\> \\ &=2\ell(y)+\<\nu_\co, 2(\rho^\vee-\rho^\vee_{J})\>=2\ell(y)+\<\nu_\co, 2\rho^\vee\>.\end{align*}

(3) Notice that $W_{J_{\bar \l}} y W_{J_{\l}}=W_{J_{\bar \l}} (y W_{J_{\l}} y\i) y =W_{J_{\bar \l}} y$ and $y \in {}^{J_{\bar \l}} W_0$, we see that $y$ is the unique minimal element of the double coset $W_{J_{\bar \l}} y W_{J_{\l}}$, that is, $y \in {}^{J_{\bar \l}} W_0 {}^{J_{\l}}$. Moreover $y W_{J_{\l}} y\i \subset W_{J_{\bar \l}}$. Thus $y$ sends simple roots of $J_{\l}$ to simple roots of $J_{\bar \l}$.
\end{proof}

\begin{prop}\label{hh}
Keep the notations in \S\ref{def-tw1}. Set $I=y I(J_e,t^{\l}w_1) y\i$. Then there exists $h \in {}^{I({J_{\bar \l}},w_2)} W_{J_{\bar \l}} {}^I$ such that

(1) $h I h\i \subset I({J_{\bar \l}}, w_2)$.

(2) $w_2=h y w_1 y\i h\i$.

(3) Both $w_2$ and $y w_1 y \i$ are of minimal lengths in their common $W_{J_{\bar \l}}$-conjugacy class.

(4) $h y \tw_1 y \i h \i \in \co_{\min}$.
\end{prop}
\begin{rmk}
By Lemma \ref{property-y} (3), $I\subset y(J_{\l}) \subset J_{\bar \l}$. Moreover, we have $y w_1 y\i \in {}^I W_0$ and $y w_1 y\i I y w_1\i y\i=I$ by the construction of $w_1$.
\end{rmk}

\begin{proof}
By Theorem \ref{partial}, there exists a minimal length element in the $W_0$-conjugacy class of $t^{\l}w_1$ of the form $t^{\bar \l}u c$, where $u \in {}^{J_{\bar \l}} (W_0 \rtimes \G)$ and $c \in W_{I({J_{\bar \l}}, u)}$. Again by Theorem \ref{partial}, there exists $c' \in W_{I({J_{\bar \l}}, u)}$ such that $uc'$ is of minimal length in the $W_{J_{\bar \l}}$-conjugacy class of $uc$. Note that $t^{\bar \l}uc$ and $t^{\bar \l}uc'$ are in the same $W_{J_{\bar \l}}$-conjugacy class. So by the choice of $t^{\bar \l}uc$, we have $$\ell(t^{\l})-\ell(u)+\ell(c)=\ell(t^{\bar \l}uc) \le \ell(t^{\bar \l}uc')=\ell(t^{\l})-\ell(u)+\ell(c'),$$ that is, $\ell(c) \le \ell(c')$. Hence $\ell(uc)=\ell(u)+\ell(c) \le \ell(u)+\ell(c')=\ell(uc')$. Therefore

(a) $u c$ is of minimal length in its $W_{J_{\bar \l}}$-conjugacy class.

By Corollary \ref{formula}, $\ell(z\i t^{\l}w_1 z)=\<\nu_\co, 2\rho^\vee\>+\<\l, 2\rho^\vee_{J}\>-\ell(w_1)$. Applying Lemma \ref{property-y}, we have $$\ell(y w_1 y\i)=2\ell(y)+\<\nu_\co, 2\rho^\vee\>+\<\l, 2\rho^\vee_{J}\>-\ell(z\i t^{\l}w_1 z).$$

On the other hand, $\ell(t^{\bar \l}uc)=\<\bar \l, 2\rho^\vee\>-\ell(u)+\ell(c)$. Hence $$\ell(uc)=\<\bar \l, 2\rho^\vee\>+2\ell(c)-\ell(t^{\bar \l}uc).$$ Since $t^{\bar \l}uc$ and $t^{\bar \l}y w_1 y\i$ are in the same $W_0$-conjugacy class, then $uc$ and $y w_1 y\i$ are in the same $W_{J_{\bar \l}}$-conjugacy class. By (a) and Lemma \ref{property-y}, we see that $$0 \le \ell(y w_1 y\i)-\ell(uc)=\ell(t^{\bar \l}uc)-\ell(z\i t^{\l}w_1 z)-2\ell(c).$$ Notice that by our construction, $t^{\bar \l}uc$ is of minimal length in the $W_0$-conjugacy class of $z\i t^{\l}w_1 z$. Hence $c=1$ and $\ell(y w_1 y\i)=\ell(u)$. By (a), both $u$ and $y w_1 y\i$ are of minimal lengths in the their common $W_{J_{\bar \l}}$-conjugacy class.

By Proposition \ref{fact}, there exists $h \in {}^{I({J_{\bar \l}},u)} W_{J_{\bar \l}} {}^I$ such that $u=h y w_1 y\i h\i$ and $h I h\i \subset I({J_{\bar \l}}, u)$. Thus $h y \tw_1 y \i h \i \in t^{\bar \l} u W_{I({J_{\bar \l}}, u)}=t^{\bar \l} u W_{I(S_0, t^{\bar \l} u)}$. The $W_0$-conjugacy class of $\tw_1$ intersects both $t^{\bar \l} u W_{I(S_0, t^{\bar \l} u)}$ and $t^{\bar \l} w_2 W_{I(S_0, t^{\bar \l} w_2)}$. By Theorem \ref{partial}, $w_2=u$.

By definition, $x_1$ is a minimal length element in the $\Ad(w_1)$-twisted conjugacy class by $W_{I(J_e, t^\l w_1)}$. Thus $h y x_1 y \i h \i$ is of minimal length in its $\Ad(w_2)$-twisted onjugacy class by $W_{h I h\i}$. By Lemma \ref{min-in-wj}, $h y x_1 y \i h \i$ is of minimal length in its $\Ad(w_2)$-twisted conjugacy class by $W_{I({J_{\bar \l}}, w_2)}$. Thus by Theorem \ref{partial}, $h y \tw_1 y \i h \i=t^{\bar \l} w_2 (h y x_1 y \i h \i)$ is of minimal length in the $W_0$-conjugacy class of $\tw'$. So $h y \tw_1 y \i h \i \in \co_{\min}$.
\end{proof}

\subsection{} Now we prove Theorem \ref{thm4.1}.

By Lemma \ref{fact''}, $\tw_0$ and $\tw_1$ are of minimal length (with respect to $\ell_{J}$) in their $\tW_{J}$-conjugacy class. Hence by $\S$\ref{min-tw}, \[\tag{a} T^{J}_{\tw_0} \equiv T_{\tw_1}^{J}=\th_\l T_{w_1 \i} \i T_{x_1} \mod [\ch_J, \ch_J].\]

Let $x'=y x_1 y\i \in W_I \subset W_{J_{\bar \l}}$ and $x''=hx'h\i \in W_{h I h\i}$. We show that
\[\tag{b} \th_{\l}T_{w_1\i}\i T_{x_1} \equiv \th_{\bar \l} T_{y w_1\i y\i}\i T_{x'} \mod[\ch, \ch].\]

Let $y=s_r \cdots s_1$ be a reduced expression. For each $k$, let $\a_k$ be the positive simple root corresponding to $s_k$ and let $\l_k=s_k \cdots s_1(\l)$. Since $y s_1 \cdots s_{k-1}(\a_k)<0$, then $$\<\l_{k-1}, \a_k^\vee\>=\<\l, s_1 \cdots s_{k-1}(\a_k^\vee)\><0.$$ By Lemma \ref{ineq'}(2), $\<\l_{k-1}, \a_k^\vee\>=-1$. By $\S$\ref{th} (6), $T_{s_k} \th_{\l_{k-1}}=\th_{\l_k}T_{s_k}\i$. Applying it successively, we have that $$T_y \th_{\l}=T_{s_r} \cdots T_{s_1} \th_{\l}=\th_{y(\l)}T_{s_r}\i \cdots T_{s_1}\i=\th_{\bar \l}T_{y\i}\i.$$

Since $y \in {}^{J_{\bar \l}} W_0$, we have $\ell(x'y)=\ell(yx_1)=\ell(x')+\ell(y)=\ell(y)+\ell(x_1)$. By Lemma \ref{property-y}, $\ell(y w_1 \i y \i)=2 \ell(y)+\ell(w_1)$. Hence $T_y T_{x_1} T_y \i=T_{x'}$ and $T_y T_{w_1 \i} T_{y \i}=T_{y w_1 \i y \i}$. Therefore \begin{align*} T_y \th_{\l} T_{w_1\i}\i T_{x_1} T_y\i&=\th_{\bar \l} T_{y \i} \i T_{w_1 \i} \i T_{x_1} T_y \i=\th_{\bar \l} (T_{y \i} \i T_{w_1 \i} \i T_y \i) (T_y T_{x_1} T_y \i) \\ &=\th_{\bar \l} T_{y w_1\i y\i}\i T_{x'}.\end{align*} (b) is proved.

Notice that $h \in W_{J_{\bar \l}}$. By $\S$\ref{th} (5), $T_h \th_{\bar \l} T_h \i=\th_{\bar \l}$. By Proposition \ref{hh},  $$\ell(y w_1 y\i h\i)=\ell(h\i w_2)=\ell(h\i)+\ell(w_2)=\ell(y w_1 y\i)+\ell(h\i).$$ Thus $T_h T_{y w_1 \i y \i}\i T_h \i=T_{w_2\i}\i$. Since $h \in {}^{I(J_{\bar \l},w_2)} W_{J_{\bar \l}} {}^I$ and $h(I) \subset I(J_{\bar \l}, w_2)$, we have that $\ell(x''h)=\ell(h x')=\ell(h)+\ell(x')=\ell(x'')+\ell(h)$ and $T_h T_{x'} T_h \i=T_{x''}$. So

\begin{align*}T_h \th_{\bar \l} T_{y w_1\i y\i} T_{x'} T_h\i &=\th_{\bar \l} (T_h T_{y w_1\i y\i}\i T_h\i) (T_h T_{x'} T_h\i) \\ &=\th_{\bar \l} T_{w_2\i}\i T_{x''}.\end{align*}

By Proposition \ref{hh}, $h y \tw_1 y \i h \i=t^{\bar \l} w_2 x''$ and $\tw'$ are both of minimal lengths in $\co$. By Theorem \ref{min} and \S\ref{min-tw}, \[\tag{c} T_{\tw'} \equiv T_{t^{\bar \l} w_2 x''}=\th_{\bar \l} T_{w_2\i}\i T_{x''} \equiv \th_{\bar \l} T_{y w_1\i y\i}\i T_{x'}  \mod [\ch, \ch].\]

Combining (a), (b) and (c), $$T_{w_\co} \equiv T_{\tw_0}^{J} \mod [\ch, \ch].$$

\begin{example} \label{eg} Let's consider the extended affine Weyl group $\tW$ associated to $GL_8$. Here $\tW \cong \ZZ^8 \rtimes \fs_8$, where the permutation group $\fs_8$ of $\{1, 2, \cdots, 8\}$ acts on $\ZZ^8$ in a natural way. Let $\tu=t^{\chi}\s$ with $\chi=[\chi_1, \cdots, \chi_8]$ and $\s \in \fs_8$. Then $$\ell(\tu)=\sum_{i<j,\s(i)<\s(j)}|\l_i-\l_j|+\sum_{i<j,\s(i)>\s(j)}|\l_i-\l_j-1|.$$

Take $\chi=[1,1,1,1,1,0,0,0] \in \ZZ^8$ and $x=(6,3,1)(7,4,8,5,2) \in \fs_8$. Then $\tw'=t^{\chi} x \in {}^{S_0} \tW$ is an minimal length element in its conjugacy class. 

Let $J=\{(1,2),(2,3),(4,5), (6,7), (7,8)\} \subset S_0$ and $\tw \in \tW_J=\ZZ^8 \rtimes W_J$ with $\l=[1,1,0,1,1,1,0,0]$ and $w=(3,2,1)(7,5,8,6,4)$. Then $\ell_J(\tw)=0$. In particular, $\tw$ is of minimal length (in the sense of $\ell_J$) in its conjugacy class of $\tW_J$.

By Theorem \ref{thm4.1}, $$T_{\tw'} \equiv T^J_{\tw}=\th_{\l}T_{w\i}\i \mod [\ch, \ch].$$

\end{example}

\subsection{}
We call an element $w \in W_0 \rtimes \G$ {\it elliptic} if $V^w \subset V^{W_0}$ and an element $\tw \in \tW$ {\it elliptic} if $p(\tw)$ is elliptic in $W_0 \rtimes \G$. By definition, if $\tw$ is elliptic, then $\nu_{\tw} \in V^{W_0}$.

A conjugacy class $\co$ in $W_0 \rtimes \G$ or $\tW$ is called {\it elliptic} if $\tw$ is elliptic for some (or, equivalently any) $\tw \in \co$.

Now we discuss the choice of $v$ in $\S$\ref{def-tw0}. If we assume furthermore that $v$ is a regular point of $V^{p(\tw')}$, then $\bar v=z(v)$ is a regular point of $V^{p(\tw_0)}$. Thus $V^{p(\tw_0)} \subset \cap_{\a \in R_J} H_{\a, 0}=V^{W_J}$. Hence $\tw_0$ is elliptic in $\tW_J$.

\subsection{} Let $\co$ be a conjugacy class and $\tw, \tw' \in \co$ with $\nu_{\tw}=\nu_{\tw'}=\nu_\co$. Let $x \in \tW$ such that $x \tw x \i=\tw'$. Then $x \in \tW_{J_\co}$. In particular, the set $\{\tw \in \co \cap \tW_{J_{\nu_\co}}; \nu_{\tw}=\nu_\co\}$ is a single $\tW_{J_{\nu_\co}}$-conjugacy class.

Let $\ca$ be the set of pairs $(J, C)$, where $J \subset S_0$, $C$ is an elliptic conjugacy class of $\tW_J$ and $\nu_{\tw}$ is dominant for some (or, equivalently any) $\tw \in C$. For any $(J, C), (J', C') \in \ca$, we write $(J, C) \sim (J', C')$ if $\nu_{\tw}=\nu_{\tw'}$ for $\tw \in C$ and $\tw' \in C'$ and there exists $x \in {}^{J'} (W_{J_{\nu_{\tw}}} \rtimes \G_{J_{\nu_{\tw}}})^J$ such that $x J x \i=J'$ and $x C x \i=C'$.

\begin{lem}
The map from $\ca$ to the set of conjugacy classes of $\tW$ sending $(J, C)$ to the unique conjugacy class $\co$ of $\tW$ with $C \subset \co$ gives a bijection from $\ca/\sim$ to the set of conjugacy classes of $\tW$.
\end{lem}

\begin{proof}
If $(J, C) \sim (J', C')$, then $C$ and $C'$ are in the same conjugacy class of $\tW$. On the other hand, suppose that $C$ and $C'$ are in the same conjugacy class $\co$. Let $\tw \in C$ and $J_{\nu_{\tw}}$. Then $\nu_{\tw} \in V^{W_J}$ and $J \subset J_{\nu_{\tw}}$. Similarly, $J' \subset J_{\nu_{\tw}}$. Then $C, C' \subset \{\tw_1 \in \co \cap \tW_{J_{\nu_{\tw}}}; \nu_{\tw_1}=\nu_\co\}$ is in the same $\tW_K$-conjugacy class. In particular, there exists $x \in W_{J_{\nu_{\tw}}} \rtimes \G_{J_{\nu_{\tw}}}$ such that $x \tw x \i \in C'$. Hence $p(x \tw x \i)$ is an elliptic element in $W_{J'} \rtimes \G_{J'}$. By \cite[Proposition 5.2]{CH}, $x=x' x_1$ for some $x' \in W_{J'}$, $x_1 \in {}^{J'} (W_{J_{\nu_{\tw}}} \rtimes \G_{J_{\nu_{\tw}}})^J$ such that $x_1 J x_1 \i=J'$. Hence $x_1 \tW_J x_1 \i=\tW_{J'}$ and $x_1 C x_1 \i=C'$.
\end{proof}

Now combining Theorem \ref{thm4.1} and Theorem \ref{IMbasis}, we have

\begin{thm}
(1) The elements $\{T^J_\co\}_{(J, \co) \in \ca/\sim}$ span $\bar \ch$ as an $\ca$-module.

(2) If $q_s^{\frac{1}{2}}=q_t^{\frac{1}{2}}$ for all $s, t \in S$, then $\{T^J_\co\}_{(J, \co) \in \ca/\sim}$ is a basis of $\bar \ch$.
\end{thm}

This gives Bernstein presentation of the cocenter $\bar \ch$.

\section{$P$-alcove elements and the cocenter of $\ch$}
\subsection{} For any $\a \in R$ and an alcove $C$, let $k(\a, C)$ be the unique integer $k$ such that $C$ lies in the region between the hyperplanes $H_{\a, k}$ and $H_{\a, k-1}$. For any alcoves $C$ and $C'$, we say that $C \ge_\a C'$ if $k(\a, C) \ge k(\a, C')$.

Let $J \subset S_0$ and $z \in W_0$. Following \cite[\S 4.1]{GHN}, we say an element $\tw \in \tW$ is a {\it $(J, z)$-alcove element}\footnote{In fact, for $\tw \in X \rtimes W_0$ and $\d \in \G$, $\tw \d$ is a $(J, z)$-alcove element if and only if $\tw C_0$ is a $(J, z\i, \d)$-alcove in \cite[\S 4.1]{GHN}. This is a generalization of the $P$-alcove introduced in \cite{GHKR}.} if

(1) $z \tw z\i \in \tW_J$ and

(2) $\tw C_0 \geqslant_{\a} C_0$ for all $\a \in z\i(R^+ - R_J^+)$.

Note that if $\tw$ is a $(J, z)$-alcove element, then it is also a $(J, u z)$-alcove element for any $u \in W_J$.

If $\tw$ is a $(J, z)$-alcove element, we may also call $\tw$ a $P$-alcove element, where $P=z \i P_J z$ is a semistandard parabolic subgroup of the connected reductive group $G$ associated to the root datum $\fR$.


\begin{lem}\label{red}
Let $\tw \in \tW$ be a $(J, z)$-alcove and let $s \in S$.

(1) If $\ell(\tw)=\ell(s \tw s)$, then $s \tw s$ is a $(J, z p(s))$-alcove element;

(2) If $\tw > s \tw s$, then $z p(s) z\i \in W_J$. Moreover, both $s \tw$ and $s \tw s$ are $(J, z)$-alcove elements.
\end{lem}

\begin{rmk}
In part (2), $s \tw$ and $s \tw s$ are also $(J, z p(s))$-alcove elements.
\end{rmk}
\begin{proof}

Part (1) is proved in \cite[Lemma 4.4.3]{GHN}.

Assume $\tw>s \tw s$ and $s=s_H$ is the reflection along $H=H_{\a,k} \in \fH$ for some $\a \in R$ and $k \in \ZZ$. By replacing $\a$ by $-\a$ if necessary, we can assume that $z(\a) \in R^+$. If $z(\a) \notin R_J$, then $\a, p(\tw)(\a) \in z\i(R^+-R_J^+)$. Note that $\tw > s \tw s$, so $H, \tw H \in \fH(C_0, \tw C_0)$. Hence $\tw C_0 >_{\a} C_0$ and $\tw C_0 >_{p(\tw)(\a)} C_0$ since $\tw$ is a $(J, z)$-alcove. Applying $\tw$ to the first inequality we have $\tw^2 C_0 >_{p(\tw)(\a)} \tw C_0$. Hence both $C_0$ and $\tw^2 C_0$ are separated from $\tw C_0$ by $\tw H$. In other words, $C_0$ and $\tw^2 C_0$ are on the same side of $\tw H$. So $\tw C_0 >_{\a} C_0$ and $\tw^2 C_0 >_{p(\tw)(\a)} \tw C_0$ can't happen at the same time. That is a contradiction. The ``moreover" part  follows from \cite[Lemma 4.4.2]{GHN}.
\end{proof}

\begin{thm}
Let $\tw \in \tW$, $J \subset S_0$ and $z \in {}^J W_0$ such that $\tw$ is a $(J, z)$-alcove. Then $$T_{\tw} \in \ch_J + [\ch, \ch].$$
\end{thm}
\begin{proof}
We argue by induction on the length of $\tw$. Suppose that $\tw$ is of minimal length in its conjugacy class. By \cite[Proposition 2.5 \& Lemma 2.7]{HN2} and Lemma \ref{red}, we may assume further that $\bar C_0$ contains a regular point of $V_{\tw}$.

Let $\mu \in V$ be a dominant vector such that $J=J_\mu$. Since ${\tw}$ is a $(J, z)$-alcove, then $z p({\tw}) z\i(\mu)=\mu$, that is, $z\i(\mu)+ V_{\tw}=V_{\tw}$. Moreover \[\tag{a} R^+-R_J^+ \subset \{\a \in R^+; \<z(\nu_{\tw}), \a^\vee\> \geqslant 0\}. \]

Let $v=\nu_{\tw}+\e z\i(\mu)$ with $\e$ a sufficiently small positive real number. We have $V_{\tw}=V_{\tw}+v$. Let $z_1=uz$ with $u \in W_J$ such that $\<z_1(v), \a^\vee\> \geqslant 0$ for each $\a \in R_J^+$. Let $\b \in R^+ - R_J^+$. By (a), $\<z_1(\nu_{\tw}), \b^\vee\>=\<z(\nu_{\tw}), u\i(\b^\vee)\> \geqslant 0$. Moreover $\<z_1 z\i(\mu), \b^\vee\>=\<\mu, u\i(\b^\vee)\> > 0$. Hence $\<z_1(v), \b^\vee\> > 0$. So $z_1(v)$ is dominant. Since $v$ lies in a sufficiently small neighborhood of $\nu_{\tw}$, $z_1(\nu_{\tw})$ is also dominant. Now applying Proposition \ref{thm4.1} (2), $T_{\tw} \in \ch_{J_{\bar \nu_{\tw}} \cap J_{\bar v}}+[\ch, \ch]$.

Let $\a \in R_{J_{\bar \nu_{\tw}} \cap J_{\bar v}}$. Then $\<z_1(v), \a^\vee\>=\<z_1(\nu_{\tw}), \a^\vee\>=0$. Hence $\<z_1 z\i(\mu), \a^\vee\>=\<\mu,\a^\vee\>=0$. Thus $J_{\bar \nu_{\tw}} \cap J_{\bar v} \subset J$. The statement holds for $\tw$.

Now assume that $\tw$ is not of minimal length in its conjugacy class and the statement holds for all $\tw' \in \tW$ with $\ell(\tw')<\ell(\tw)$.

By Theorem \ref{min}, there exist $\tw_1 \cong \tw$ and $s \in S$ such that $\ell(s \tw_1 s)<\ell(\tw_1)=\ell(\tw)$. Then \[T_{\tw} \equiv T_{\tw_1} \equiv T_{s \tw_1 s}+(q_s^{\frac{1}{2}}-q_s^{-\frac{1}{2}}) T_{s \tw_1} \mod [\ch, \ch].\] Here $\ell(s \tw_1 s), \ell(s \tw_1)<\ell(\tw)$. By Lemma \ref{red}, $\tw_1, s \tw_1 s, s \tw_1$ are $(J, z_1)$-alcove elements for some $z_1 \in {}^J W_0$. The statement follows from induction hypothesis.
\end{proof}

\subsection{} We introduce the class polynomials, following \cite[Theorem 5.3]{HN2}. Suppose that $q_s^{\frac{1}{2}}=q_t^{\frac{1}{2}}$ for all $s, t \in S$. We simply write $v$ for $q_s^{\frac{1}{2}}$. In this case, the parameter function $p_t^{\frac{1}{2}}$ in $\S$\ref{ptj} also equals to $v$.

Let $\tw \in \tW$. Then for any conjugacy class $\co$ of $\tW$, there exists a polynomial $f_{\tw, \co} \in \ZZ[v-v \i]$ with nonnegative coefficient such that $f_{\tw, \co}$ is nonzero only for finitely many $\co$ and \[\tag{a} T_{\tw} \equiv \sum_{\co} f_{\tw, \co} T_{\co} \mod [\tH, \tH].\] 

The polynomials can be constructed explicitly as follows.

If $\tw$ is a minimal element in a conjugacy class of $\tW$, then we set $f_{\tw, \co}=\begin{cases} 1, & \text{ if } \tw \in \co \\ 0, & \text{ if } \tw \notin \co \end{cases}$. Suppose that $\tw$ is not a minimal element in its conjugacy class and that for any $\tw' \in \tW$ with $\ell(\tw')<\ell(\tw)$, $f_{\tw', \co}$ is already defined. By Theorem \ref{min}, there exist $\tw_1 \approx \tw$ and $s \in S$ such that $\ell(s \tw_1 s)<\ell(\tw_1)=\ell(\tw)$. In this case, $\ell(s \tw)<\ell(\tw)$ and we define $f_{\tw, \co}$ as $$f_{\tw, \co}=(v_s-v_s \i) f_{s \tw_1, \co}+f_{s \tw_1 s, \co}.$$

\begin{thm}\label{compare-class}
Let $\tw \in \tW$, $J \subset S_0$ and $z \in {}^J W_0$ such that $\tw$ is a $(J, z)$-alcove. Let \begin{align*}T_{\tw} &\equiv \sum_{\co} f_{\tw, \co}T_{w_\co} \mod [\ch, \ch]; \\ T_{z \tw z\i}^J &\equiv \sum_{\co'} f_{z \tw z\i, \co'}^J T_{w_{\co'}}^J \mod [\ch_J, \ch_J],\end{align*} where $\co$ and $\co'$ run over all the conjugacy classes of $\tW$ and $\tW_J$ respectively in the above summations. Then $$f_{\tw, \co}=\sum_{\co' \subset \co} f_{z \tw z\i, \co'}^J.$$
\end{thm}

\begin{proof}
We argue by induction on the length of $\tw$. If $\tw$ is of minimal length in its conjugacy class, then by Proposition \ref{mm}, $z \tw z \i$ is also a minimal length element (with respect to $\ell_J$) in its $\tW_J$-conjugacy class. The statement holds in this case.

Now assume that $\tw$ is not of minimal length in its conjugacy class and the statement holds for all $\tw' \in \tW$ with $\ell(\tw')<\ell(\tw)$.

By Theorem \ref{min}, there exist $\tw_1 \cong \tw$ and $s \in S$ such that $\ell(s \tw_1 s)<\ell(\tw_1)=\ell(\tw)$. By Corollary \ref{cor4.4} and Lemma \ref{red}, there exists $z_1 \in {}^J W_0$ such that $\tw_1, s \tw_1 s, s \tw_1$ are $(J, z_1)$-alcove elements and $z \tw z \i \cong z_1 \tw_1 z_1 \i$ with respect to $\tW_J$.

Let $t=z s z\i$. Then by Lemma \ref{red} and Lemma \ref{zs}, $t \in \tW_J$ and $\ell_J(t)=1$. By the proof of Corollary \ref{cor4.4}, $\ell_J(t z_1 \tw_1 z_1 \i t \i)<\ell_J(z_1 \tw_1 z_1 \i)$. So by the construction of class polynomials,
\begin{gather*}
f_{\tw, \co}=f_{\tw_1, \co}=(v-v \i) f_{s \tw_1, \co}+f_{s \tw_1 s, \co}; \\
f^J_{z \tw z \i, \co'}=f^J_{z_1 \tw_1 z_1 \i, \co'}=(v-v \i) f^J_{t z_1 \tw_1 z_1 \i, \co'}+f^J_{t z_1 \tw_1 z_1 \i t\i, \co'}.
\end{gather*}
The statement follows from induction hypothesis.
\end{proof}

\subsection{} In the rest of this section, we discuss some application to affine Deligne-Lusztig varieties.

Let $\mathbb F_q$ be the finite field with $q$ elements. Let $k$ be an algebraic closure of $\mathbb F_q$. Let $F= \mathbb F_q( (\e))$, the field of Laurent series over $\mathbb F_q$, and $L=k( (\e))$, the field of Laurent series over $k$.

Let $G$ be a quasi-split connected reductive group over $F$ which splits over a tamely ramified extension of $F$.  Let $\s$ be the Frobenius automorphism of $L/F$. We denote the induced automorphism on $G(L)$ also by $\s$.

Let $\ci$ be a $\s$-invariant Iwahori subgroup of $G(L)$. The $\ci$-double cosets in $G(L)$ are parameterized by the extended affine Weyl group $W_G$. The automorphism on $W_G$ induced by $\s$ is denoted by $\d$. Set $\tW=W_G \rtimes \<\d\>$.

For $\tw \in W_G$ and $b \in G(L)$, set
\[
X_{\tw}(b) = \{ g \ci \in G(L)/\ci;\ g^{-1}b\sigma(g) \in \ci \tw \ci \}.
\]
This is the affine Deligne-Lusztig variety attached to $\tw$ and $b$. It plays an important role in arithmetic geometry. We refer to \cite{GHKR}, \cite{GHN} and \cite{H99} for further information.

The relation between the affine Deligne-Lusztig varieties and the class polynomials of the associated affine Hecke algebra is found in \cite[Theorem 6.1]{H99}.

\begin{thm}\label{deg}
Let $b \in G(L)$ and $\tw \in \tW$. Then
\[\dim (X_{\tw}(b))=\max_{\co} \frac{1}{2}(\ell(\tw)+\ell(w_\co)+\deg(f_{\tw \d, \co}))-\<\bar \nu_b, 2\rho\>,\] where $\co$ ranges over the $\tW$-conjugacy class of $W_G \d \subset \tW$ such that $\nu_\co$ equals the Newton point of $b$ and $\k_G(x)=\k_G(b)$ for some (or equivalently, any) $x \in W_G$ with $x \d \in \co$. Here $\k_G$ is the Kottwitz map \cite{Ko}.
\end{thm}

\subsection{} For $J \subset S_0$, let $M_J$ be the corresponding Levi subgroup of $G$ defined in \cite[3.2]{GHN} and $\k_J$ the Kottwitz map for $M_J(L)$. As a consequence of Theorem \ref{compare-class}, we have

\begin{thm}
Let $\tw \in W_G$ and $z \in W_0$. Suppose $\tw \d$ is a $(J, z)$-alcove element. Then for any $b \in M_J(L)$, $X_{\tw}(b)=\emptyset$ unless $\k_J(z \tw \d(z)\i)=\k_J(b)$.
\end{thm}

\begin{rmk}
This result was first proved in \cite[Theorem 1.1.2]{GHKR} for split groups and then generalized to tamely ramified groups in \cite[Corollary 3.6.1]{GHN}. The approach there is geometric, using Moy-Prasad filtration. The approach here is more algebraic.
\end{rmk}

\begin{proof}
Assume $X_{\tw}(b) \neq \emptyset$. By Theorem \ref{deg}, there exists a conjugacy class $\co$ of $W_G \d$ such that $f_{\tw \d, \co} \neq 0$, $\nu_\co=\bar \nu_b$ and $\k_G(b)=\k_G(x)$ for some (or equivalently, any) $x \in W_G$ with $x \d \in \co$. By Theorem \ref{compare-class}, there exists a $\tW_J$-conjugacy class $\co' \subset \co$ such that $f_{z \tw \d z\i, \co'}^J \neq 0$. Choose $b' \in M_J(L)$ such that $\nu_{b'}=\nu_{\co'}$ and $\k_J(b')=\k_J(x')$ for some (or equivalently, any) $x' \in W_{M_J}$ with $x'\d \in \co'$. By \cite[Proposition 3.5.1]{GHN},
$b$ and $b'$ belong to the same $\s$-conjugacy class of $M_J(L)$. Since the affine Deligne-Lusztig variety $X_{z \tw \d(z)\i}^{M_J}(b')$ for $M_J$ is nonempty, we have $\k_J(z \tw \d(z)\i)=\k_J(b')=\k_J(b)$.
\end{proof}

\end{document}